\documentclass[11pt]{amsart}
\usepackage[margin=1.4in]{geometry}

\usepackage{graphicx}
\usepackage{commath}
\usepackage{stmaryrd}
\usepackage{xcolor}
\usepackage{bm}
\usepackage{hyperref}
\usepackage{mathabx}
\usepackage{tikz-cd}
\usepackage{enumitem}
\usepackage{caption}
\usepackage{amssymb}
\setlist[enumerate]{label=\arabic*., ref=\arabic*}

\newtheorem{lemma}{Lemma}

\newtheorem{theorem}{Theorem}
\theoremstyle{definition}
\newtheorem{definition}{Definition}

\newtheorem{proposition}{Proposition}

\newcommand{\sys}{\mathrm{sys}}
\newcommand{\M}{\mathcal{M}}
\newcommand{\R}{\mathbb{R}}
\newcommand{\G}{\mathcal{G}}
\newcommand{\Z}{\mathbb{Z}}

\renewcommand{\P}{\mathbb{P}}
\newcommand{\B}{\mathcal{B}}
\newcommand{\T}{\mathbb{T}}
\renewcommand{\H}{\mathbb{H}}

\title{A note on surfaces with large systoles}

\author{Yifei Cai}
\address{Qiuzhen College, Tsinghua University, Beijing 100084, China}
\email{caiyf23@mails.tsinghua.edu.cn}

\begin{document}

\maketitle

\begin{abstract}
     We show that for every sufficiently large genus $g$, there exists a closed hyperbolic surface $S_g$ with systole  $\sys(S_g)\geq \log g-12\log\log g$. In particular, $$\liminf_{g\to \infty}\frac{\max\{\sys(S):S\in \M_g\}}{\log g}\geq 1,$$ improving the previously known bound $2/9$.
     
     This note is a continuation of our previous work on the diameter of finite covers \cite{CL}, using the same framework of constant-twist pants decomposition to study systoles. The proof was developed by GPT-5.6 Sol through an extended discussion with the author.
\end{abstract}
\section{Introduction}

The systole $\sys(S)$ of a hyperbolic surface $S$ is the length of the shortest closed geodesic in $S$. It is known that the systole function $\sys:\M_g\to \R_{>0}$ attains a maximum \cite{Mum71} on the moduli space $\M_g$. A standard area argument shows that this maximum has an upper bound $$\max_{S\in \M_g}\sys(S)\leq 2\log(4g-2).$$

It is still widely open to determine the asymptotic behavior of this maximum. Brooks \cite{Br88} and Buser--Sarnak \cite{BS94} proved that $$\limsup_{g\to \infty}\frac{\max\{\sys(S):S\in \M_g\}}{\log g}\geq \frac{4}{3}.$$

Buser--Sarnak also showed that the corresponding limit inferior is positive. Recently, Katz--Sabourau \cite{KS25} obtained an explicit every-genus lower bound:
$$\liminf_{g\to \infty}\frac{\max\{\sys(S):S\in \M_g\}}{\log g}\geq\frac{19}{120},$$
and more recently, Liu--Petri \cite{LP23}, using a random construction, improved this to $\frac{2}{9}$.

In this note, we further improve the constant to 1. The main result is the following

\begin{theorem}\label{main}
    For every sufficiently large $g$, there exists a closed hyperbolic surface $S_g\in \M_g$ satisfying $$\sys(S_g)\geq \log g-12\log \log g.$$
    In particular,$$\liminf_{g\to \infty}\frac{\max\{\sys(S):S\in \M_g\}}{\log g}\geq 1.$$
\end{theorem}

\subsection{Graph-theoretic motivation}
The problem of finding a surface with a large systole has a strong graph-theoretical parallel. The analog of the systole is the \emph{girth} of a graph, which is by definition the length of its shortest cycle. For a $k$-regular graph $\G$ with $n$ vertices, the Moore bound gives $$\mathrm{girth}(\G)\leq (2+o(1))\log_{k-1}(n).$$ For cubic graphs, Biggs--Hoare \cite{BH83} introduced the family of cubic sextet graphs, and Weiss \cite{W84} proved that they satisfy $$\mathrm{girth}(\G)\geq \left(\frac{4}{3}+o(1)\right)\log_2(n).$$
On the other hand, Erd\"os--Sachs \cite{ES63} proved combinatorially that for \emph{every} sufficiently large even $n$, there exists a $k$-regular graph $\G$ on $n$ vertices, with $$\mathrm{girth}(\G)\geq (1-o(1))\log_{k-1}(n).$$
More recently, Linial--Simkin \cite{LS21} recovered this using a random greedy algorithm.

We have already seen a striking parallel between constants arising in graph theory and hyperbolic geometry. In particular, Petri--Walker \cite{PW}, using a construction inspired by Erd\"os--Sachs, produced a sequence of closed hyperbolic surfaces $S_{g_i}$ satisfying
$$\sys(S_{g_i})\geq (1-o(1))\log g_i.$$

Thus the constant 1 is already achievable along a sequence. The result of the present note shows that the same constant can be achieved in every sufficiently large genus.

However, there are problems if we directly transfer the graph-theoretic result to hyperbolic geometry. A closed geodesic may cross a large number of pairs of pants while the hyperbolic length remains short. Hence a high-girth pants decomposition graph does not automatically give a surface with a large systole.

\subsection{Proof strategy}

Let $n=2g-2$, and let $a=4\log n$.

Take $n$ identical hyperbolic pairs of pants $P_a$ with lengths of all three cuffs equal to $a$. To specify the gluing, first take a cycle with vertex set $\Z/n\Z$, and choose uniformly randomly a fixed-point free involution $$\iota:\Z/nZ\to \Z/n\Z, \ \iota^2=id,\ \iota(i)\ne i,$$ 
Adding the edges $(i,\iota(i))$ gives a random trivalent graph $\G$. Gluing the pants $P_a$ according to $\G$, with \emph{constant twists 1} along each cuff, gives a random closed hyperbolic surface $S_\iota$ of genus $g$. See Figure \ref{fig1} for an illustration.

\begin{figure}[h]
    \centering
    \includegraphics[width=0.6\linewidth]{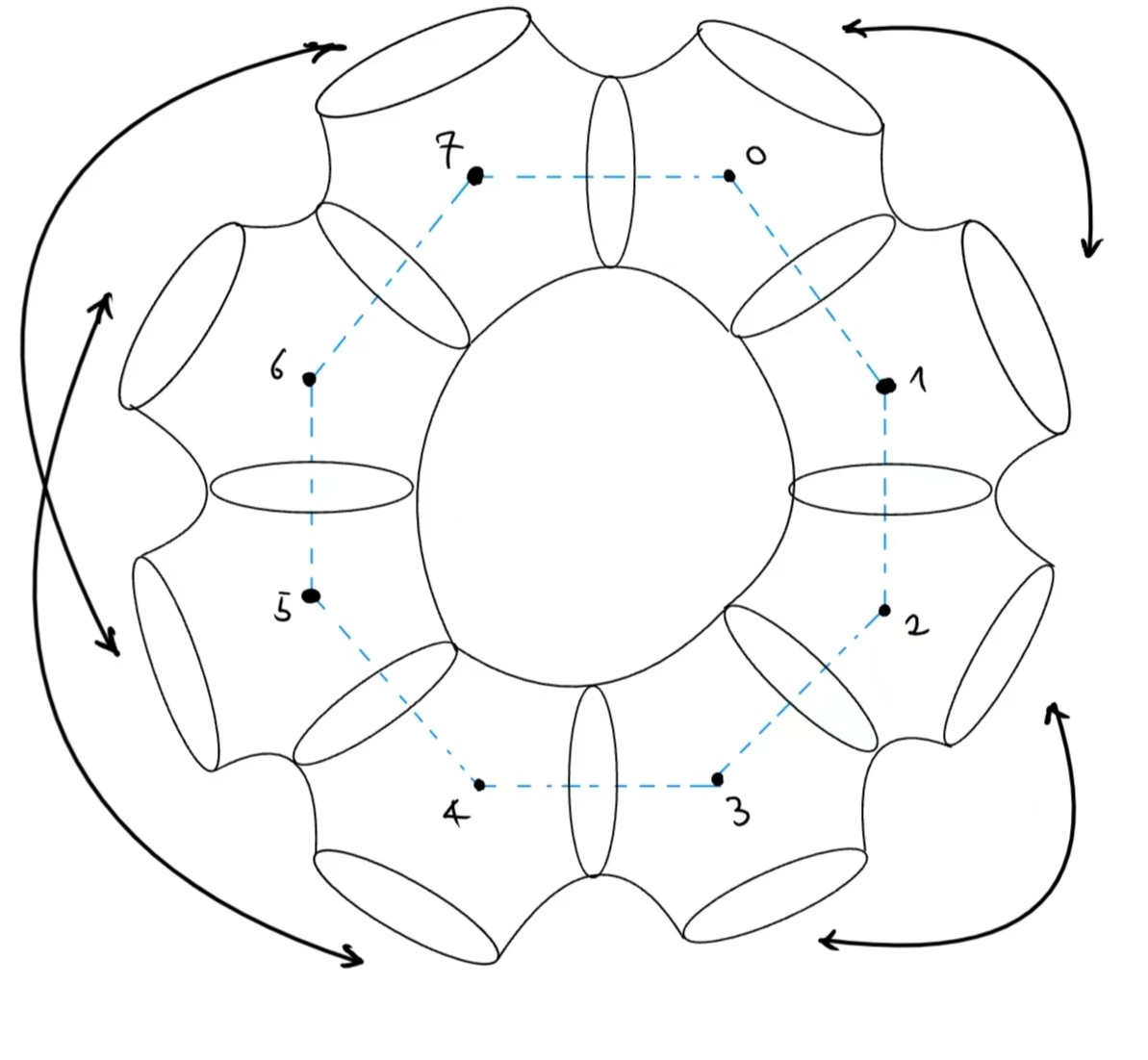}
    \caption{An illustration for the construction of $S_\iota$. Here $n=8$ and $\iota$ is taken to be $(01)(23)(46)(57)$.}
    \label{fig1}
\end{figure}

Set $$ R=\log n-12\log \log n.$$ We will show that with positive probability, the random surface $S_\iota$ contains no closed geodesic of length at most $R$.

First we encode short geodesic segments. Work in a fixed decorated (infinite) pants tree $T_a$, obtained by gluing $P_a$ along an infinite trivalent tree, with twists 1. Every surface $S_\iota$ defined above is a quotient of this same $T_a$; only the quotient map depends on $\iota$.

Fix a lift $\widetilde{c_0}$ of the initial cuff $c_0\subset S_\iota$. A geodesic segment in $S_\iota$ starting from $c_0$ lifts to a geodesic segment in $T$ starting from $\widetilde{c_0}$. We record the \emph{coding} of this segment by the terminal cuff $\widetilde{c}$ it reaches, which fully determines the combinatorial path from $\widetilde{c_0}$ to $\widetilde{c}$. A geometric estimate shows that in $T$, the number of reachable terminal cuffs from $\widetilde{c_0}$ via a segment of length at most $R$ is $$O(a^2e^R).$$ Hence the number of possible codings is $O(a^2e^R)$.

Also, using a Kahn--Markovi\'c estimate \cite{KM12}, the nonzero twists allow us to effectively compare hyperbolic distance and combinatorial distance:  a geodesic segment of length $\leq R$ crosses only $O(aR)$ cuffs, hence can only have coding of combinatorial length $l=O(aR)$. 

For a fixed coding of combinatorial length $l=O(aR)$, the probability that the random involution closes the corresponding path down in the surface is $O(l^3 n^{-1})$. Therefore, the probability that there exists a closed geodesic of length at most $R$ \emph{passing through $c_0$} is $$O\left(\frac{a^5R^3e^R}{n}\right)= O\left(\frac{1}{(\log n)^{4}}\right).$$

To get the final conclusion, a union bound over $n$ cuffs is insufficient. The crucial additional point is locality. A short geodesic meets only $O(aR)$ cuffs, and the corresponding bad event (i.e. short-geodesic event) can interact only with the bad events involving one of these cuffs. More precisely, the total probability mass of the bad events conflicting with a given bad event is small, and the Lov\'asz Local Lemma allows all bad events to be avoided simultaneously. This yields a surface with systole $>R$.

\subsection{Declaration on the use of AI}
Starting from the constant-twist pants decomposition approach described in this note, GPT-5.6 Sol (OpenAI) developed the first complete proof of the main theorem through an extended discussion with the author. The proof in this manuscript is checked, simplified and reorganized by the author. The author takes full responsibility for the content and correctness of this manuscript.

\subsection{Acknowledgments} The author would like to thank Qiliang Luo for many insightful discussions on this problem. The author would also like to thank Yi Huang for his support.

\section{Random model and the Lovász Local Lemma}

\subsection{Random involutions and bad partial matchings} Fix $n$ to be even. Let $\Omega_n$ denote the set of all fixed-point free involutions $$\iota:\Z/nZ\to \Z/n\Z, \ \iota^2=id,\ \iota(i)\ne i.$$
Equip $\Omega_n$ with the uniform distribution. Thus an element $\iota\in \Omega_n$ divides the $n$ elements of $\Z/n\Z$ into $n/2$ disjoint pairs. We first introduce a partial version of this matching.

\begin{definition}[Partial matching]
    A \emph{partial matching} $F$ is a collection of disjoint pairs of elements in $\Z/n\Z$. For a partial matching $$F=\{\{x_1,y_1\},\cdots,\{x_k,y_k\}\}, $$
    we say that an involution $\iota$ respects $F$, denoted by $F\subset \iota$, if for every $1\leq i\leq k$, $\iota(x_i)=y_i$. Moreover, define the \emph{canonical event} associated to $F$ as $$A_F=\{\iota\in \Omega_n:F\subset \iota\}.$$
\end{definition}
The probability of a canonical event can be calculated explicitly.
\begin{lemma}
    If $|F|=k=o(\sqrt{n})$, then the probability of its associated canonical event is
    $$\P(A_F)=\frac{(n-2k-1)!!}{(n-1)!!}\sim n^{-k}.$$
\end{lemma}

\begin{proof}
The number of fixed-point free involutions is $$|\Omega_n|=(n-1)!!.$$ 

If $|F|=k$, then after $2k$ elements in $\Z/n\Z$ are matched, the remaining $(n-2k)$ elements can be matched arbitrarily. Hence $$\P(A_F)=\frac{(n-2k-1)!!}{(n-1)!!}\leq (n-2k+1)^{-k}.$$The latter is $n^{-k}\exp(O(\frac{k^2}{n}))$, so $\P(A_F)\sim n^{-k}$ if $k=o(\sqrt{n})$.
\end{proof}

We next give them a geometric interpretation. Recall that our surface is obtained by gluing pairs of pants of cuff lengths $a$, with constant twist 1. First perform the gluings corresponding to the fixed $n$-cycle, and let $S_\varnothing$ denote the resulting hyperbolic surface, with $n$ boundary components. Exactly one cuff of each pair of pants remains unglued, and we denote them by $c_0,c_1,\cdots,c_{n-1}$. For a partial matching $$F=\{\{x_1,y_1\},\cdots,\{x_k,y_k\}\}, $$ let $S_F$ be the hyperbolic surface (possibly with boundary), obtained from $S_\varnothing$ by gluing $c_{x_i}$ to $c_{y_i}$ with twists 1. Similarly define $S_\iota$ for an involution $\iota\in \Omega_n$.

Fix a target length $R=\log n-12\log\log n$ as in the introduction.
\begin{definition}[Bad partial matchings]
    A partial matching $F$ is called \emph{$R$-bad,} if $S_F$ contains a closed geodesic of length at most $R$. 
    
    An $R$-bad partial matching $F$ is called \emph{minimal,} if no proper submatching of $F$ is $R$-bad. Let $\B_R$ denote the set of all minimal $R$-bad partial matchings.
\end{definition}

If $S_\iota$ contains a closed geodesic of length at most $R$, then it contains a partial matching $F\subset \iota$ which is minimal $R$-bad. So $$\{\iota\in \Omega_n:\sys(S_\iota)\leq R\}\subset \bigcup_{F\in \B_R}A_F.$$

In order to prove Theorem \ref{main}, it therefore suffices to show \begin{equation}\label{posprob}
    \P(\cap_{F\in \B_R}\overline{A_F})>0.
\end{equation}
\subsection{Conflict graphs and negative dependence}

In order to prove (\ref{posprob}), we introduce a local lemma in this subsection.

\begin{definition}[Conflict graph]
    Two partial matchings $F_1,F_2$ \emph{conflict,} if $F_1\cup F_2$ is not a partial matching. 
    
    Let $\mathcal{F}$ be a collection of partial matchings. We define a graph $G=G(\mathcal{F}),$ called \emph{conflict graph} of $\mathcal{F}$, as follows: 
    \begin{enumerate}
        \item $V(G)=\mathcal{F}$,
        \item $E(G)=\{(F_1,F_2):F_1\in \mathcal{F} \text{ and } F_2\in \mathcal{F} \text{ conflict}\}.$
    \end{enumerate}
\end{definition}
It turns out that the conflict graphs satisfy a good property.
\begin{definition}[Negative dependency graph]
    Let $A_1,\cdots, A_m$ be events. A graph $G$ on $\{1,\cdots,m\}$ is called a \emph{negative dependency graph,} if $$\P(A_i|\cap_{j\in S}\overline{A_j})\leq \P(A_i),$$for any index $i$ and subset $S\subset\{j:(i,j)\notin E(G)\}$, whenever the conditional probability is defined.
\end{definition}

\begin{theorem}[{\cite[Theorem 3]{Lu-Szekely-14}}]
    Let $\mathcal{F}$ be any collection of nonempty partial matchings, then the conflict graph $G(\mathcal{F})$ defined above is a negative dependency graph for $\{A_F\}_{F\in \mathcal{F}}$.
\end{theorem}

The main machine of the proof is the next lopsided Lov\'asz Local Lemma, which was first introduced in \cite{ES91}, and we use the following formulation from {\cite[Lemma 1]{Lu-Szekely-14}}.

\begin{theorem}[Lopsided Lov\'asz Local Lemma]\label{LLL}
    Let $A_1,\cdots,A_m$ be events with a negative dependency graph $G$. If there exist numbers $x_1,\cdots,x_m\in [0,1)$ such that \begin{equation}\label{LLLhypo}
        \P(A_i)\leq x_i\prod_{(i,j)\in E(G)}(1-x_j),
    \end{equation}for every $i$, then $$\P(\cap_{i=1}^m \overline{A_i})\geq \prod_{i=1}^m(1-x_i)>0.$$
\end{theorem}

\subsection{A local criterion}

Let $\mathcal{F}$ be any collection of nonempty partial matchings. We use here Theorem \ref{LLL} for the conflict graph $G=G(\mathcal{F})$ to deduce a useful criterion. For $i\in \Z/n\Z$, define $$L_i(\mathcal{F})=\sum_{F\in\mathcal{F},\ i\in \mathrm{supp}(F)}\P(A_F),$$
    where $\mathrm{supp}(F)$ is the set of $2|F|$ elements touched by $F$. Also write \begin{equation}\label{defKL}
        K=\max_{F\in \mathcal{F}}|F|,\ \ L=\max_{i\in \Z/n\Z}L_i(\mathcal{F}).
    \end{equation}

\begin{proposition}\label{finalaim}
    If $8KL\leq 1$, then $$\P(\cap_{F\in \mathcal{F}}\overline{A_F})>0.$$
\end{proposition}
\begin{proof}
    Write $p_F=\P(A_F)$ and set $x_F=2p_F$. 
    
    If $F'$ and $F$ conflict, then $\mathrm{supp}(F)\cap \mathrm{supp}(F')\ne \varnothing$. Therefore $$\sum_{(F,F')\in E(G)}p_{F'}\leq \sum_{i\in \mathrm{supp}(F)}\sum_{\substack {F'\in \mathcal{F},\\ i\in \mathrm{supp}(F')}}p_{F'}\leq \sum_{i\in \mathrm{supp}(F)}L_i(\mathcal{F})\leq \frac{1}{4},$$by assumption. Therefore $$\prod_{(F,F')\in E(G)}(1-x_{F'})\geq 1-\sum_{(F,F')\in E(G)}x_{F'}\geq \frac{1}{2}.$$ Consequently $$x_F\prod_{(F,F')\in E(G)}(1-x_{F'})\geq p_F=\P(A_F),$$ which is exactly the hypothesis (\ref{LLLhypo}). The Lov\'asz Local Lemma then gives the positive probability.
\end{proof}

We will apply this proposition for $\mathcal{F}=\B_R$. Thus the proof of Theorem \ref{main} is reduced to two estimates for $K$ and $L$. Sections 3 and 4 establish these two bounds, respectively.

\section{Short geodesics and the codings}
In this section we introduce a fixed infinite pants tree which simultaneously covers all the random surface $S_{\iota}$. This allows us to encode geometrically short curves in a space which is independent of the choice of random involutions. We then use the Kahn--Markovi\'c estimate to bound both the combinatorial complexity and the number of possible codings of such curves.
\subsection{Pants tree cover}

Let $P_a$ be the hyperbolic pair of pants whose three cuffs have lengths $a$. Let
 $$\T=\mathrm{Cay}\left(\langle r,m|m^2=1\rangle\right)$$ be the infinite trivalent tree, whose vertices are represented by reduced words in symbols $r,r^{-1},m$. We now put a copy $P_w$ of $P_a$ over all vertex $w\in V(\T)$ and glue them along the edges, with constant twist 1. Denote the resulting \emph{pants tree} by $T_a$.

Now we explain how $T_a$ covers $S_\iota$. The $r^{\pm1}$-edges encode the deterministic cycle edge, while the $m$-edges encode the random matching edges.

Let $\tau:\Z/n\Z\to \Z/n\Z, \tau(k)=k+1$ be the translation. On the graph level the map is given by $$\phi_{\iota}:r^{q_1}mr^{q_2}\cdots mr^{q_k}\mapsto \tau^{q_1}\iota \tau^{q_2}\cdots \iota\tau^{q_k}(i), $$where $i\in \Z/n\Z$ is arbitrarily fixed. Since the pairs of pants and the gluing maps are the same on corresponding edges, the graph cover $\phi_\iota $ induces a surface cover $$\pi_\iota:T_a\to S_\iota.$$

\subsection{Geometric complexity of short geodesics}
We recall two geometric properties of the pants tree $T_a$.

First, for a short geodesic segment in $T_a$, its combinatorial path is non-backtracking.

\begin{lemma}\label{nonback}
    If $\beta$ is a geodesic segment of length $<a/2$, then its projection to $\T$ is non-backtracking.
\end{lemma}

\begin{proof}
    If it is backtracking, then $\beta$ must contain a segment which is a segment in a pair of pants $P_a$ with its endpoints in the same cuff. 

    This segment has length larger than that of an orthogeodesic, whose length $d$ must satisfy a standard right-angle pentagon formula \cite[Theorem 2.3.4]{Buser}: $$\sinh\frac{a}{4}\sinh\frac{d}{2}=\cosh\frac{a}{2}.$$
    
    A simple calculation then implies $d>a/2,$ and hence the lemma follows.
\end{proof}

Second, we will use Kahn--Markovi\'c's estimate to give a uniform comparison between the hyperbolic and combinatorial distances. Let $\Lambda$ be the lift of the union of all cuffs in $T_a$ to $\H^2$.

\begin{lemma}[{\cite[Lemma 2.2]{KM12}}]\label{KM}
    Let $\sigma$ be a geodesic segment in $\H^2$ of length less than $e^{-5}$, which transversely intersects $\Lambda$. Then there exists a universal constant $C_1>0$, such that $|\sigma\cap \Lambda|\leq C_1a.$
\end{lemma}

It is the nonzero twist that allows Lemma \ref{KM} to be correct: the lengths of the seams in $P_a$ behave like $e^{-a/4}$, and the twists prevent the geodesic from always moving along the seams.

As a result, there exists a constant $C_0>0$, such that every geodesic segment in $\H^2$ of length $s$ can cross at most $C_0as$ cuff lifts. We first apply this estimate to the partially glued surface $S_\varnothing$.

\begin{lemma}\label{init}
    For all sufficiently large $n$, $$\sys(S_\varnothing)>R.$$
\end{lemma}
\begin{proof}
Suppose that $S_\varnothing $ contains a closed geodesic $\gamma$ with length at most $ R$. Since $a>R$, $\gamma$ cannot be contained in a single pair of pants.

Then Lemma \ref{nonback} indicates that the combinatorial path of $\gamma$ is a closed non-backtracking walk on $n$-cycle. In particular, it must cross at least $n$ cuffs. However, by the estimate following Lemma \ref{KM}, $\gamma$ crosses at most $C_0aR=O((\log n)^2)$ cuffs, a contradiction.
\end{proof}

We then get the first quantity needed for Proposition \ref{finalaim}.

\begin{proposition}\label{combleng}
    There exists a universal constant $C_0>0$, such that every $F\in \B_R$ satisfies: $$1\leq |F|\leq C_0aR.$$
\end{proposition}
\begin{proof}
By Lemma \ref{init}, the empty partial matching is not $R$-bad, so $$\min_{F\in \B_R}|F|\geq 1.$$  

For $F\in \B_R$, let $\gamma\subset S_F$ be a closed geodesic of length at most $R$. We claim that $\gamma$ crosses every cuff arising from a pair in $F$. Indeed, if it did not cross some such cuff $e$, then ungluing that cuff would leave $\gamma$ unchanged as a closed geodesic in $S_{F-e}$, contradicting the minimality of $F$. 

Thus $|F|$ is at most the number of cuff crossings of $\gamma$. Again, the estimate following Lemma \ref{KM} gives the result.
\end{proof}

\subsection{Coding geodesic segments}
We now introduce a coding for short geodesic segments in the fixed pants tree $T_a$.

Recall that the underlying trivalent tree $\T$ has vertices represented by reduced words in $r,r^{-1},m$. Fix the $m$-edge incident to the vertex 1, and let $\widetilde{c_0}$ denote the corresponding cuff in $T_a$.

\begin{definition}[Codings]
    Suppose that a geodesic segment $\beta\subset T_a$ of length at most $R$ starts on $\widetilde c_0$ and terminates on another cuff $\widetilde c$. We call $\widetilde c$ the \emph{coding} of $\beta$.
\end{definition}
 Since $R<a/2$, Lemma \ref{nonback} implies that the projection of $\beta$ to $\T$ is non-backtracking. In particular, the coding uniquely determines the successive pants it encounters, and hence the edge types encountered along the path.

Let $$W_s(\widetilde{c_0})=\{\widetilde{c}\text{ cuff of } T_a:d_{hyp}(\widetilde{c_0},\widetilde{c})\leq s\}.$$
Thus every geodesic segment of length at most $s$ starting from $\widetilde{c_0}$ has its coding in $W_s(\widetilde{c_0})$.

The following estimate gives the main geometric counting input.
\begin{proposition}
    There exists a constant $C>0$, such that for every cuff $\widetilde{c_0}$ and every $s\geq 1$,
    $$|W_s(\widetilde{c_0})|\leq Ca^2e^s.$$

\end{proposition}
\begin{proof}
    Let $q:\H^2\to T_a$ be the universal covering. Let $\widehat{c_0}$ be a lift of $\widetilde{c_0}$, and let $I\subset \widehat{c_0}$ be a fundamental interval of its stabilizer. So $\ell(I)=a$.

    We first observe that there exists a sufficiently small $r_0>0$, such that every ball $B(x,r_0)\subset\mathbb H^2$ meets at most $C_1a$ cuff lifts. Indeed, we can choose a geodesic polygon containing $B(x,r_0)$, with a uniformly bounded number of sides, each of length less than $e^{-5}$. Every complete cuff lift meeting the ball must cross the boundary of the polygon, and Lemma \ref{KM} bounds the number crossing each side by $O(a)$.

    Now consider the $s$-neighborhood $N_s(I)$ of $I$. Choose a maximal $r_0$-separated set of points in $N_s(I)$. The balls of radius $r_0/2$ centered at these points are disjoint and lie in $N_{s+r_0/2}(I).$ The area of $N_{s+r_0/2}(I)$ is $$\mathrm{Area}\left(N_{s+r_0/2}(I)\right)=2a\sinh (s+\frac{r_0}{2})+2\pi\left(\cosh(s+ \frac{r_0}{2}) -1\right)\leq C_2ae^s,$$for a universal $C_2$. Comparing areas shows that the number of points is at most $C_3ae^s$. By maximality, the balls of radius $r_0$ centered at these points cover $N_s(I)$.

    By the previous observation, each ball meets at most $C_1a$ cuff lifts. Therefore $|W_s(\widetilde{c_0})|\leq C_1C_3a^2e^s$.
\end{proof}

Both the coding $\widetilde c$ and its associated combinatorial path
are objects in the fixed pants tree $T_a$, and are therefore
independent of the random involution. The randomness enters only when
this fixed path is projected to $S_\iota$.

\section{Closing Probabilities and Proof of Theorem \ref{main}}

We now estimate the probability that a \emph{fixed} coding in $T_a$ closes after projecting to the random surface $S_{\iota}$.

\subsection{Closing codings}

Fix $i\in \Z/n\Z$, and choose the quotient map such that $i$ is the quotient of the vertex 1 in $\T$.

\begin{definition}[Closing codings]
    We say that a coding $\widetilde{c}$ \emph{closes} from $i$, if after the projection to the pants decomposition graph of $S_{\iota}$, it is the same as the projection of $\widetilde{c_0}$.
\end{definition}

\begin{proposition}
    There exists a universal constant $C>0$, such that for any coding $\widetilde{c}$ of combinatorial length at most $l=o(n)$, $$\P(\widetilde{c}\text{ closes from }i)\leq \frac{Cl^3}{n}.$$
\end{proposition}

\begin{proof}

    Let $$1=v_0,v_1,\cdots, v_s,\ s\leq l $$ be the vertices on the fixed combinatorial path in $\T$ associated to $\widetilde{c}$. For any given involution $\iota$, write $$x_{j}=\phi_{\iota}(v_j)\in \Z/n\Z$$ for the projected vertex. So an $r^{\pm 1}$-edge sends $x_j$ to $x_{j+1}=x_j\pm1$, and an $m$-edge sends $x_j$ to $x_{j+1}=\iota(x_j)$.

    At each new $m$-step (i.e. not exposed yet), conditional on all previous exposures, the new vertex $y=\iota(x_j)$ is uniform among the vertices that remain unmatched up to this step. We call this $m$-step \emph{bad}, if $$y=\iota(x_j)=\tau^q(x_k), \ \text{for some }k\leq j, \ |q|\leq l.$$ Conditional on the preceding exposures,
$$\mathbb P(\text{the $m$-step is bad})
\leq
\frac{(l+1)(2l+1)}{n-2l}=\frac{O(l^2)}{n},$$
for all sufficiently large $n$.

We claim that if the coding $\widetilde{c}$ closes from $i$, then some new $m$-step is bad. If not, consider the last new $m$-edge in the coding, and write its endpoint as $y$. There can be no later $m$-step: indeed, before the next $m$-step the path moves only by $r^{\pm1}$-steps, so its current vertex has the form $\tau^q(y),\ |q|\leq l$. If that $m$-step were already exposed, then $\tau^q(y)=x $ for some previously visited $x_k$, and hence $y=\tau^{-q}(x_k)$, contradicting that the last new $m$-step was not bad.

Thus the path contains only $r^{\pm1}$-steps after $y$. If the coding closes, it must return to $x_0=i$, so
$$i=\tau^q(y)$$
for some $|q|\leq l$, again contradicting that the last new $m$-step is not bad.

Since there are at most $l+1$ new $m$-steps, a union bound gives $$\P(\widetilde{c}\text{ closed from }i)\leq \frac{Cl^3}{n},$$for some universal constant $C>0$.
\end{proof}

\subsection{From closing codings to bad partial matchings}
Recall that the quotient map $\phi_\iota$ is chosen such that $\phi_{\iota}(1)=i$, and $\widetilde{c_0}$ is the $m$-edge in $\T$ incident to $1$.

For a minimal $R$-bad partial matching $F\in \B_R$, such that $i\in \mathrm{supp}(F)$, we can choose a closed geodesic $\gamma\subset S_F$ of length $\ell(\gamma)\leq R$. By minimality, $\gamma$ crosses every cuff arising from a pair in $F$. In particular, it crosses the projection of $\widetilde{c_0}$. We may cut $\gamma$ at the crossing and lift it to $T_a$. Then it must end on a cuff $\widetilde{c}\in W_R(\widetilde{c_0})$. Thus we associate to each $F\in\B_R$ once and for all a coding $\widetilde{c}(F)=\widetilde{c}$. By definition,
\begin{equation}\label{1}
    A_F\subset \{\widetilde{c}(F)\text{ closes from }i\}.
\end{equation}

A fixed coding $\widetilde{c}$ may correspond to different minimal $R$-bad partial matchings. Write $\mathcal{F}(\widetilde{c})=\{F\in \B_R: i\in\mathrm{supp}(F),\ \widetilde{c}(F)=\widetilde{c}\}$.

\begin{lemma}\label{indep}
    If $F_1,F_2$ are two distinct elements in $\mathcal{F}(\widetilde{c})$, then $$A_{F_1}\cap A_{F_2}=\varnothing.$$
\end{lemma}

\begin{proof}
    For $F_i\in \mathcal{F}(\widetilde{c})$, let $\gamma_{F_i}$ be the closed geodesic used to define the coding $\widetilde{c}(F_i)=\widetilde{c}$. Then by minimality, the collection of matching pairs whose corresponding cuffs are crossed by $\gamma_{F_i}$ is exactly $F_i$.

    Suppose that there exists an $\iota\in A_{F_1}\cap A_{F_2}$, which is equivalent to saying that $F_1\cup F_2\subset \iota$. Since $\iota$ is given and extends $F_1$, the combinatorial path of $\widetilde{c}$ projected to $S_{\iota}$ is fixed, and that coincides exactly with that of $\gamma_{F_1}$. The same is true for $F_2$. Therefore, the combinatorial paths of $\gamma_{F_1}$ and $\gamma_{F_2}$ coincide, and in particular, the collection of matching pairs they cross are the same, so $F_1=F_2$.
\end{proof}

As a consequence of (\ref{1}) and Lemma \ref{indep},  $$\sum_{F\in \mathcal{F}(\widetilde{c})}\P(A_{F})\leq \P(\widetilde{c}\text{ closes from }i).$$
Finally, a union bound over all possible $\widetilde{c}$ gives \begin{equation}\label{2}
    L_i(\B_R)=\sum_{\substack{F\in\B_R,\\i\in \mathrm{supp}(F)}}\P(A_F)\leq \sum_{\widetilde{c}\in W_R(\widetilde{c_0})}\P(\widetilde{c}\text{ closes from }i)\leq \frac{Ca^2l^3e^R}{n},
\end{equation}
where $l$ can be chosen to be $C_0aR$.

\subsection{Proof of Theorem \ref{main}}

Recall that all we need to prove is (\ref{posprob}):
$$\P(\cap_{F\in \B_R}\overline{A_F})>0.$$

 First, by Lemma \ref{init} the empty partial matching $F=\varnothing$ is not $R$-bad, so we can apply Proposition \ref{finalaim} for $\mathcal{F}=\B_R$. We want to prove that the quantities $K$ and $L$ defined in (\ref{defKL}) satisfy $8KL\leq 1$.

Combining (\ref{2}) with Proposition \ref{combleng}, $$L=\max_{i}L_i(\B_R)\leq \frac{Ca^2l^3e^R}{n}=O\left(\frac{1}{(\log n)^{4}}\right).$$
and by Proposition \ref{combleng},
$$K=\max_{F\in \B_R}|F|\leq caR=O((\log n)^2).$$

Letting $n\to \infty$, we have $8KL\to 0$. This completes the proof of Theorem \ref{main}.

\bibliography{references}
\bibliographystyle{alpha}


    


    

    

\end{document}